\documentclass[11pt]{amsart}

\usepackage[T1]{fontenc}
\usepackage[utf8]{inputenc}
\usepackage{amsmath,amssymb,amsthm,mathtools}
\usepackage{xcolor}
\usepackage{array,booktabs,microtype,float}
\usepackage[a4paper,margin=32mm]{geometry}
\usepackage[colorlinks=true,linkcolor=blue!55!black,
  citecolor=blue!55!black,urlcolor=blue!55!black]{hyperref}

\newtheorem{theorem}{Theorem}[section]
\newtheorem{proposition}[theorem]{Proposition}
\newtheorem{lemma}[theorem]{Lemma}

\numberwithin{equation}{section}
\theoremstyle{definition}
\newtheorem{example}[theorem]{Example}

\theoremstyle{remark}

\newcommand{\C}{\mathbb C}
\newcommand{\PP}{\mathbb P}
\newcommand{\Gr}{\operatorname{Gr}}
\newcommand{\LG}{\operatorname{LG}}

\newcommand{\Lie}{\operatorname{Lie}}
\newcommand{\Sym}{\operatorname{Sym}}
\newcommand{\Aut}{\operatorname{Aut}}

\newcommand{\rank}{\operatorname{rank}}

\newcommand{\cC}{\mathcal C}

\newcommand{\bP}{\mathbb{P}}

\newcommand{\fg}{\mathfrak{g}}
\newcommand{\fh}{\mathfrak{h}}

\newcommand{\fl}{\mathfrak{l}}

\newcommand{\fn}{\mathfrak{n}}
\newcommand{\fp}{\mathfrak{p}}

\newcommand{\fu}{\mathfrak{u}}

\title[Heisenberg compactifications of homogeneous varieties]
{Heisenberg Equivariant Compactifications\\
of Rational Homogeneous Varieties}
\author{Cong Ding, Baohua Fu, and Zhijun Luo}

\address{Cong Ding, School of Mathematical Sciences, Shenzhen University, 518060, Guangdong, China}
\email{\href{congding@szu.edu.cn }{congding@szu.edu.cn }}

\address{Baohua Fu, State Key Laboratory of Mathematical Sciences, Morningside Center of Mathematics, Academy of Mathematics and Systems Science, Chinese Academy of Sciences, Beijing 100190, China; and School of Mathematical Sciences, University of Chinese Academy of Sciences, Beijing, China}
\email{\href{bhfu@math.ac.cn}{bhfu@math.ac.cn}}
\urladdr{\href{http://www.math.ac.cn/people/fbh/}{http://www.math.ac.cn/people/fbh/}}

\address{Zhijun Luo, Center for Complex Geometry, Institute for Basic Science, 55 Expo-ro, Yuseong-gu, Daejeon, 34126, Republic of Korea.}
\email{\href{luozj@amss.ac.cn}{luozj@amss.ac.cn}, \href{luozj@ibs.re.kr}{luozj@ibs.re.kr}}
\date{}
\subjclass[2020]{14M17, 14L30, 17B22}
\keywords{rational homogeneous variety, Heisenberg group, equivariant
compactification, adjoint variety, cominuscule varieties}

\begin{document}
\begin{abstract}
Let $G/P$ be a complex projective rational homogeneous variety of dimension $2m+1$. We prove that $G/P$ is an equivariant compactification of the Heisenberg group of dimension $2m+1$ if and only if it is isomorphic to either an adjoint variety, or the 3-dimensional smooth quadric $Q^3$, or a product $\bP^{2m+1-d} \times Y$ with $1 \leq d=\dim Y \leq m$, where $Y$ is a product of cominuscule varieties.
\end{abstract}
\maketitle

\section{Introduction}

We work over $\C$. An \emph{equivariant compactification} of a
connected affine algebraic group $H$ is a complete smooth variety
endowed with an algebraic $H$-action having a Zariski-open orbit
equivariantly isomorphic to $H$, with $H$ acting on itself by left
translations. This notion includes several classical theories. For a torus, one obtains toric varieties, while compactifications of reductive groups and, more generally, embeddings of homogeneous spaces are closely related to spherical varieties and the Luna--Vust theory (see for example \cite{Timashev}).

On the other hand, equivariant compactifications of unipotent groups do not yet have a comparable general embedding theory. 
The simplest case is the vector group $\mathbb G_a^n$. Hassett and Tschinkel (\cite{HassettTschinkel}) related its equivariant compactifications by projective space to finite-dimensional local
algebras. Arzhantsev proved (\cite[Theorem~1]{Arzhantsev}) that rational homogeneous varieties which are equivariant compactifications of $\mathbb G_a^n$ are products of cominuscule varieties (also called irreducible Hermitian symmetric spaces of compact type, which are listed in Table~\ref{tab:ihss-list}). There is a broad class of equivariant compactifications of $\mathbb G_a^n$ through Euler-symmetric projective varieties (\cite[Section~5]{FHEuler}).
These results motivate the corresponding problem for 
noncommutative unipotent groups.

In this paper, we consider equivariant compactifications of Heisenberg groups, which will also be called Heisenberg equivariant compactifications. 
Recall that a Heisenberg group of dimension $2m+1$ is the connected unipotent group whose Lie algebra $\fh$ is given by $\fh=V\oplus\C z$, where $V$ is a vector space of dimension $2m$ endowed with a nondegenerate alternating form $\omega$, with the following Lie brackets:
\[
\forall v, w \in V, \quad 
  [v,w]=\omega(v,w)z,\qquad [\fh,z]=0.
\]
We allow $m=0$; in this case the one-dimensional Heisenberg group is, by convention, the vector group $\mathbb G_a$.
Equivalently, after fixing the form $\omega$ on $V$, the Heisenberg group $H$ is the variety $V\times\C$ with multiplication
\[
 (u,t) \cdot (v,s)=
 \left(u+v,t+s+\frac12\omega(u,v)\right).
\]

A fundamental family of examples of Heisenberg equivariant compactifications is given by adjoint varieties. If $\fg$ is a simple Lie algebra of rank at least $2$, its highest-root grading is a contact grading
\[
 \fg=\fg_{-2}\oplus\fg_{-1}\oplus\fg_0\oplus\fg_1\oplus\fg_2,
 \qquad \dim\fg_{-2}=1.
\]
The negative part $\fg_{-}:=\fg_{-2} \oplus \fg_{-1}$ is a Heisenberg algebra, and its Lie group is a Heisenberg group which acts simply transitively on the open Bruhat cell of the adjoint variety in $\mathbb{P} \fg$, which makes the adjoint variety a Heisenberg equivariant compactification. 
The type-$A_1$ adjoint variety is $\PP^1$, which is the unique equivariant compactification of the 1-dimensional Heisenberg group $\mathbb{G}_a$.
A further example of Heisenberg equivariant compactifications is the three-dimensional quadric $Q^3$ (cf. Proposition \ref{prop:quadrics}). 

Our main theorem provides a complete classification of Heisenberg equivariant compactifications of $G/P$ as follows:

\begin{theorem}\label{thm:main}
Let $X=G/P$ be a  projective rational homogeneous variety of dimension $2m+1$. Then $X$ is a Heisenberg equivariant compactification if and only if one of the following holds:

(i) $X$ is isomorphic to an adjoint variety or $X \simeq Q^3$;

(ii) $X \simeq \bP^{2m+1-d} \times Y$, where $Y$ is a product of cominuscule varieties and   $1\leq d=\dim Y \leq m$.
\end{theorem}
Note that $\PP^{2m+1}$ is cominuscule for its full connected automorphism
group $\operatorname{PGL}_{2m+2}$, but it is the adjoint variety of
$\operatorname{Sp}_{2m+2}$, so the group defining an adjoint presentation of $X$ need not to be ${\rm Aut}^0(X).$

\begin{example} 
\label{ex:product}
The following example illustrates the product case.
Let $H$ be the 3-dimensional Heisenberg group represented by matrices
\[
 h(a,b,c)=
 \begin{pmatrix}
 1&a&c\\
 0&1&b\\
 0&0&1
 \end{pmatrix}.
\]
Let $H$ act on $\PP^1$ by
$[s:t]\mapsto[s+at:t]$, and on $\PP^2$ through matrix multiplications on column vectors. For
$x=\bigl([0:1],[0:0:1]\bigr)\in\PP^1\times\PP^2,$
the stabilizer of the first component is $\{a=0\}$, while the
stabilizer of the second is $\{b=c=0\}$. Their intersection is
trivial. Hence the  orbit $H\cdot x$ is open and
isomorphic to $H$.
\end{example}

Motivated by the uniqueness results for equivariant compactifications of $\mathbb G_a^n$ by Fano manifolds of Picard number one
(\cite[Theorem~1.2 and Corollary~1.3]{FuHwang}), Cheong proved in \cite[Theorem~1.1]{Cheong} that the compactification of the unipotent radical of a parabolic by the associated flag variety is unique for most projective rational homogeneous varieties, which in particular implies that the Heisenberg equivariant compactification structure is unique on adjoint varieties that are not projective spaces.
By contrast, Ding and Luo showed in \cite[Theorem~3.12]{DingLuo} that
$\bP^{2m+1} (m \geq 1)$ carries infinitely many inequivalent
Heisenberg compactification structures. 
We can prove that the Heisenberg equivariant compactification structure is unique on $Q^3$, but it is more subtle for the case of products. We leave this question to the reader.

In parallel to Euler-symmetric varieties (\cite{FHEuler}), the notion of Heisenberg-symmetric varieties is introduced in \cite{FuHwangContact} through contact fundamental forms, which provides many (possibly singular) Heisenberg equivariant compactifications. It remains a very interesting problem to find (or even classify) more examples of smooth Heisenberg equivariant compactifications of Picard number one.

Here is the outline of the proof of Theorem \ref{thm:main}. 
For $X=G/P$, we may assume $G = \Aut^0(X)$ up to replacing $G/P$ by another representative for $X$. We first deal with the case $G$ simple.
If $X$ is an equivariant compactification of a Heisenberg group $H$, then $H$ is a subgroup of $G$, whose Lie algebra gives a complement to $\fp =\Lie(P)$. If $\fg_{-}$ is the negative part of the fundamental grading of $\fg$ associated to $\fp$, then we prove that $\dim [\fg_{-}, \fg_{-}] \leq \dim [\fh, \fh]=1$. This implies that $G/P$ is either cominuscule or adjoint by a result of Yamaguchi (\cite{Ya}). It remains to classify which cominuscule varieties are Heisenberg equivariant compactifications.
As $H$ is unipotent, it is contained in the unipotent radical $U$ of a Borel subgroup up to a conjugation, which gives $\fh \subset \fu=\Lie(U) = \fn \ltimes W$, 
where $\fn=\fu \cap \fp$ and $W = \fg_{-1}$.
As $\fh$ is a complement to $\fp$, the inclusion $\fh \subset \fu $ is in fact given by a linear map $f: W \to \fn$. The Heisenberg bracket on $\fh$ then gives strong restrictions on $f$, which we call graph equations. We go through the list of cominuscule varieties by studying these restrictions to prove that only $Q^3$ or some coincidences of cominuscule/adjoint varieties can appear.

For general $G$, we can decompose $X=G/P =\prod_{i} G_i/P_i$ into a product with each $G_i$ simple, and we can assume there are at least two factors. We will show that there exists a unique factor $X_{i_0}$ on which the Heisenberg group $H$ acts faithfully and the other factors are all cominuscule varieties. Then a detailed analysis on the isotropy representation shows that $X_{i_0}$ is a projective space.

\medskip

The article is organized as follows. Section~2 proves the adjoint--cominuscule reduction. Section~3 establishes the graph description and proves the cominuscule classification case by case, which proves Theorem \ref{thm:main} in the case where $G$ is simple. 
Section~4 considers the product case and proves Theorem \ref{thm:main}.

\medskip

{\em Acknowledgements:} We would like to thank Jun-Muk Hwang for suggesting the problem and for helpful discussions. Based on a manuscript dated in 2025 proving Theorem \ref{thm:main} for $G$ simple via a more computational method, the authors used ChatGPT this July for exploratory discussions and language editing which led to the current version. All mathematical arguments and references were independently verified by the authors, who take full responsibility for the content. C. Ding was supported by a start-up funding from Shenzhen University and the Shenzhen Peacock Plan. B. Fu was supported by the National Key Research and Development Program of China (No. 2025YFA1017302). Z. Luo was supported by the Institute for Basic Science (Grant No. IBS-R032-D1-2026-a00). 

\section{Reduction to cominuscule varieties}

In this section $\fg$ is a complex simple Lie algebra and
$\fp\subsetneq\fg$ is a parabolic subalgebra. Fix a Cartan subalgebra
and a system $\Delta$ of simple roots such that $\fp$ is determined by
a nonempty subset $I\subset\Delta$. For a root
$\beta=\sum_i n_i\alpha_i$, put
$\deg_I(\beta)=\sum_{\alpha_i\in I}n_i$. For $j\ne0$, let
$\fg_j$ be the sum of the root spaces $\fg_\beta$ with
$\deg_I(\beta)=j$, and let $\fg_0$ be the sum of the Cartan subalgebra
and the root spaces of degree zero. Then $\fp=\bigoplus_{j\geq0}\fg_j$ and
$\fg=\bigoplus_{j=-k}^{k}\fg_j$ is called the fundamental grading of $\fg$ associated to $\fp$, where $\fg_k\ne0$. The integer $k$ is called
the depth of the grading, and we write
$\fg_-:=\bigoplus_{j<0}\fg_j$.


\begin{lemma}\label{lem:filtered-complement}
 If a Lie subalgebra
$\fl\subset\fg$ is a vector-space complement of $\fp$, then
$\fl$ has a finite filtration such that
  $\operatorname{gr}\fl\simeq\fg_-$
as graded Lie algebras. Moreover, 
\begin{equation}\label{eq:filtered-derived-bound}
  \dim[\fg_-,\fg_-]\le \dim[\fl,\fl].
\end{equation}
\end{lemma}

\begin{proof}
For every integer $r$, put $\fg^r=\bigoplus_{j\ge r}\fg_j$
with the evident extensions outside the range $[-k,k]$, and filter
$\fl$ by $F^r\fl=\fl\cap\fg^r$. Projection along $\fp$
identifies $\fl$ linearly with $\fg/\fp \simeq \fg_-$, so $\fl$ is the graph of a linear map $\phi:\fg_-\to\fp$. For $-k\le r\le-1$, we have
\[
 F^r\fl=
 \left\{u+\phi(u):u\in\bigoplus_{j=r}^{-1}\fg_j\right\}.
\]
Taking the component of degree $r$ induces
\begin{equation}\label{eq:filtered-graded-piece}
  F^r\fl/F^{r+1}\fl\simeq\fg_r.
\end{equation}
The filtration is compatible with the Lie bracket because
$[\fg^r,\fg^s]\subset\fg^{r+s}$ for all $r,s$.
If two filtered elements have the lowest components
$u_r\in\fg_r$ and $v_s\in\fg_s$, then the component of degree
$r+s$ of their bracket is $[u_r,v_s]$. Every term involving
$\phi$, or a higher negative component, has a strictly larger degree.
Thus \eqref{eq:filtered-graded-piece} identifies the associated graded bracket with the bracket on $\fg_-$.

Let $D=[\fl,\fl]$ with its induced filtration. The associated graded space
$\operatorname{gr}D$ is a graded subspace of
$\operatorname{gr}\fl$, and every bracket of two symbols in
$\operatorname{gr}\fl$ is the symbol of an element of $ \operatorname{gr}D$.
It then follows that $ [\operatorname{gr}\fl,\operatorname{gr}\fl]
 \subseteq\operatorname{gr}D.$
Taking dimensions proves \eqref{eq:filtered-derived-bound}.
\end{proof}





\begin{lemma}\label{lem:parabolic-dichotomy}
If $\dim[\fg_-,\fg_-]\le1,$ then one of the following holds:
\begin{enumerate}
\item $\fg_-$ is abelian and
      $G/P$ is a cominuscule variety;
\item $[\fg_-,\fg_-]$ is one-dimensional and $G/P$ is the
      adjoint variety.
\end{enumerate}
\end{lemma}

\begin{proof}
It is well-known (see for example \cite[Lemma 3.8]{Ya}) that
the negative part $\fg_-$ is generated as a Lie algebra by
$\fg_{-1}$, and $[\fg_-,\fg_-]=\bigoplus_{j\le-2}\fg_j$.

If $[\fg_-,\fg_-]$ is zero, the grading has depth one. Hence, the 
subset $I$ consists of a single simple root having coefficient one in the
highest root. It follows that $\fp$ is maximal cominuscule.

Suppose that $\dim [\fg_-,\fg_-]=1$. Then the grading has depth two and
$\dim\fg_{-2}=1$; thus, it is a contact grading. By
\cite[Theorem~4.1]{Ya}, the contact grading of a complex simple Lie
algebra is unique up to conjugacy and is the grading determined by the stabilizer of the highest-root line in $\PP \fg$. Hence $G/P$ is the
adjoint variety.
\end{proof}

\begin{proposition}
\label{prop:reduction}
If $X=G/P$ is a Heisenberg equivariant compactification with $G$ simple, then   $X$ is
either an adjoint variety or a cominuscule variety.
\end{proposition}

\begin{proof}
 By the Demazure--Onishchik
theorem~\cite{Demazure,Onishchik}, we may assume that $X = G/P$ satisfies $G=\Aut^0(X)$ up to taking a different representative of $X$. As $H$ acts on $X$ faithfully, $H$ can be regarded as a subgroup of $G$.
Choose $x$ in the open $H$-orbit and replace $P$ by the conjugate stabilizer $G_x$. The differential of the orbit map is a surjection
$ \fh\longrightarrow T_xX\simeq\fg/\fp,$ which is then an
isomorphism because $\dim H = \dim X$. Thus we obtain
$ \fg=\fh\oplus\fp$.
Applying Lemma~\ref{lem:filtered-complement} to the fundamental grading of $\fp$,  we obtain $\dim[\fg_-,\fg_-]\le\dim[\fh,\fh]=1.$
Now we can apply Lemma~\ref{lem:parabolic-dichotomy} to conclude.
\end{proof}


\section{The classification for simple $G$}

This section is devoted to the proof of Theorem \ref{thm:main} in the case where $G$ is simple. By Proposition \ref{prop:reduction}, we may assume that $X=G/P$ is a cominuscule variety with $G=\Aut^0(X)$ and $X$ is an equivariant compactification of a Heisenberg group
$H$.

\subsection{Graph equations} \label{sect:graph}

We start with the following result.

\begin{lemma}\label{lem:standard-position}
Let $K$ be a connected reductive group, let $Q\subset K$ be a
parabolic subgroup, and let $N\subset K$ be a connected unipotent
subgroup having an open orbit on $K/Q$. After conjugating $N$ in $K$,
there exist a Borel subgroup $B=TU$ and a $T$-fixed point
$o\in K/Q$ such that $N\subset U$ and $o$ belongs to the open
$N$-orbit.
\end{lemma}

\begin{proof}
Choose a maximal unipotent subgroup $U$ containing $N$, and let
$B=N_K(U)=TU$ be the corresponding Borel subgroup. If $x$ belongs to
the open $N$-orbit, then $B\cdot x$ is open. By the Bruhat
decomposition, it is the unique open $B$-orbit in $K/Q$, it contains
a unique $T$-fixed point $o$, and it equals $U\cdot o$. Choose
$u\in U$ with $u\cdot x=o$. Since $u\in U$, we have
$uNu^{-1}\subset U$, and the open orbit of $uNu^{-1}$ contains $o$.
\end{proof}


By Lemma~\ref{lem:standard-position}, after conjugating $H$ we may
choose a Borel subgroup $B=TU$ and a $T$-fixed point $o\in X$ such
that $H\subset U$ and $o$ belongs to the open $H$-orbit

Let $\fu=\Lie(U)$ and $\fp=\Lie(G_o).$
Apply the grading convention of Section~2 directly to the
cominuscule parabolic $G_o$.  Since the grading has depth one, it takes the form
\begin{equation}\label{eq:cominuscule-grading}
 \fg=\fg_{-1}\oplus\fg_0\oplus\fg_1,
 \qquad
 \fp=\fg_0\oplus\fg_1.
\end{equation}
Let $W:=\fg_{-1}$ and  $ \fn:=\fu\cap\fg_0$.
Then $W$ is the abelian cominuscule module.
 

\begin{proposition}\label{prop:graph}
Under the above notations, we have $\fu=\fn\ltimes W$ and 
there is a unique linear map $f:W\to\fn$ such that
$\fh=\{f(x)+x:x\in W\}.$
Moreover, there exist $0\ne z\in W$ and
$\omega\in\Lambda^2W^*$ such that for all $x,y\in W$, the following graph equations hold:
\begin{align}
\ker\omega &=\C z, \label{eq:graph-radical} \\
 f(x)y-f(y)x&=\omega(x,y)z,
     \label{eq:graph-torsion}\\
 [f(x),f(y)]&=\omega(x,y)f(z).
     \label{eq:graph-closure}
\end{align}
Here $f(x)y:=[f(x),y]\in W$.
\end{proposition}

\begin{proof}
First, let us prove that $\fu=\fn\ltimes W.$
Since $o$ belongs to the open $H$-orbit,
the differential of the orbit map at the identity gives an isomorphism $ \fh\xrightarrow{\ \sim\ }\fg/\fp.$
By \eqref{eq:cominuscule-grading}, $\fg/\fp\simeq\fg_{-1}=W.$
Because $G_o$ is opposite to the cominuscule parabolic containing
$B$, the root spaces of $\fu$ have grading degree either $0$ or
$-1$. Their degree-zero part is $\fn=\fu\cap\fg_0$, while their
degree-$-1$ part is all of $\fg_{-1}=W$. Since
$[\fg_0,\fg_{-1}]\subset\fg_{-1}$ and
$[\fg_{-1},\fg_{-1}]=0$, this proves $\fu=\fn\ltimes W.$

Under the identification $\fg/\fp\simeq W$, the quotient map
$\fu\to\fg/\fp$ is simply the projection $\fn\oplus W\longrightarrow W$.
Its restriction to $\fh$ is the isomorphism $\fh \simeq \fg/\fp$. Hence, $\fh$ is the graph of a unique linear map $f: W\to\fn$.

Transport the bracket of $\fh$ to $W$ by this projection. Since $\fh$ is Heisenberg, its central line becomes $\C z\subset W$, and the transported bracket has the form
\[
 [x,y]_{\fh}=\omega(x,y)z
\]
for an alternating form $\omega$ satisfying
\eqref{eq:graph-radical}. Since $W=\fg_{-1}$ is abelian,
\[
 [f(x)+x,f(y)+y]
 =[f(x),f(y)]+f(x)y-f(y)x.
\]
On the other hand, the graph description gives
\[
 [f(x)+x,f(y)+y]
 =\omega(x,y)\bigl(f(z)+z\bigr).
\]
Comparing the $W$- and $\fn$-components yields
\eqref{eq:graph-torsion} and \eqref{eq:graph-closure}, respectively.
\end{proof}


\subsection{Cominuscule varieties}
We now list the cominuscule varieties, their cominuscule modules $W=\fg_{-1}$, and
their semisimple Levi algebras $\fl=[\fg_0,\fg_0]$ in
Table~\ref{tab:ihss-list} (see
\cite[Section~3.1]{LandsbergManivelRHV}). Here
$\Delta_{10}^{\pm}$ denotes a half-spin module and
$J_3(\mathbb O_{\C})$ the complex Albert algebra.

\begin{table}[H]
\centering
\caption{Cominuscule varieties and their isotropy representations.}
\label{tab:ihss-list}
\renewcommand{\arraystretch}{1.08}
\begin{tabular}{@{}>{$}l<{$}>{$}l<{$}>{$}l<{$}>{$}l<{$}@{}}
\toprule
X&W&\fl&\dim X\\
\midrule
\Gr(k,N)&
 (\C^{k})^*\otimes\C^{N-k}&
 \mathfrak{sl}_k\oplus\mathfrak{sl}_{N-k}&k(N-k)\\
Q^n&\C^n&\mathfrak{so}_n&n\\
\LG(r,2r)&
 \Sym^2 (\C^{r})^*&\mathfrak{sl}_r&r(r+1)/2\\
\mathbb S_r&
 \Lambda^2 (\C^{r})^*&\mathfrak{sl}_r&r(r-1)/2\\
E_6/P_1,\ E_6/P_6&
 \Delta_{10}^{\pm}&\mathfrak{so}_{10}&16\\
E_7/P_7&
 J_3(\mathbb O_{\C})&\mathfrak e_6&27\\
\bottomrule
\end{tabular}
\end{table}

\begin{lemma} \label{lem:rank-one}
Let $X=G/P$ be a cominuscule variety with
$G=\Aut^0(X)$ and $W=\fg_{-1}$ the abelian cominuscule $\fg_0$-module.
 If $X$ is not a projective space, then every nonzero nilpotent element in the image of  $\fg_0\longrightarrow\mathfrak{gl}(W)$
has rank at least two.
\end{lemma}
\begin{proof}

Let $\rho:\fg_0\to\mathfrak{gl}(W)$ denote the isotropy
representation, which is faithful. 

Write $\fg_0=\fl\oplus\C E$, where
$\fl=[\fg_0,\fg_0]$ and $E$ acts on $W$ by a
nonzero scalar. Let $v=\rho(s)+cI_W \in \mathfrak{gl}(W)$ be nilpotent, with
$s\in\fl$. Since every representation of the semisimple Lie
algebra $\fl$ has trace zero,
$0=\operatorname{tr} v=c\dim W$; thus $c=0$. The faithfulness of $\rho$
and compatibility with Jordan decomposition then imply that $s$ is a
nilpotent element of $\fl$.

Write $\fl=\bigoplus_j\fl_j$ as a sum of simple ideals,
and write $s=(s_j)$. Choose $i$ with $s_i\ne0$. Let
$\mathcal O_{\min,i}$ be the minimal nonzero nilpotent orbit in
$\fl_i$. The standard closure order on nilpotent orbits gives
$\mathcal O_{\min,i}\subset\overline{\mathcal O_{s_i}}$, whereas
$0\in\overline{\mathcal O_{s_j}}$ for $j\ne i$. So if
$e_{\theta_i}$ is a highest-long-root vector of $\fl_i$, then
\[
 (0,\ldots,e_{\theta_i},\ldots,0)
 \in\overline{\mathcal O_s}.
\]
For every $r$, the determinantal locus
$\{T\in\mathfrak{gl}(W):\rank T\leq r\}$ is Zariski closed. Therefore
\[
 \rank\rho(e_{\theta_i})\leq\rank\rho(s).
\]
It remains to compute the rank of a highest-long-root vector in each
simple Levi factor. The values, in the notation of
Table~\ref{tab:ihss-list}, are given in Table~\ref{tab:root-operator-ranks}.
\begin{table}[H]
\centering
\caption{Least ranks of nonzero nilpotent root operators in the isotropy
representations.}
\label{tab:root-operator-ranks}
\renewcommand{\arraystretch}{1.08}
\begin{tabular}{@{}>{$}l<{$}>{$}l<{$}>{$}l<{$}>{$}l<{$}@{}}
\toprule
X&W&\fl&\text{least }\rank\rho(e_\theta)\\
\midrule
\Gr(k,N),\ 2\le k\le N-2&
 (\C^k)^*\otimes\C^{N-k}&
 \mathfrak{sl}_k\oplus\mathfrak{sl}_{N-k}&\min\{k,N-k\}\\
Q^n,\ n\ge3&\C^n&\mathfrak{so}_n&2\\
\LG(r,2r),\ r\ge2&
 \Sym^2(\C^r)^*&\mathfrak{sl}_r&r\\
\mathbb S_r,\ r\ge4&
 \Lambda^2(\C^r)^*&\mathfrak{sl}_r&r-2\\
E_6/P_1,\ E_6/P_6&
 \Delta_{10}^{\pm}&\mathfrak{so}_{10}&4\\
E_7/P_7&
 J_3(\mathbb O_{\C})&\mathfrak e_6&6\\
\bottomrule
\end{tabular}
\end{table}
For the classical tensor representations, these numbers follow by
applying a root matrix $E_{12}$. Its rank on
$(\C^k)^*\otimes\C^{N-k}$ is $N-k$ or $k$, according to the simple
factor; its ranks on $\Sym^2(\C^r)^*$ and $\Lambda^2(\C^r)^*$ are
respectively $r$ and $r-2$. A long-root operator on the standard orthogonal
module has rank two. For the half-spin representation, use
$S^+=\Lambda^{\mathrm{even}}\C^5$: exterior multiplication by
$e_1\wedge e_2$ has rank $1+\binom{3}{2}=4$. In the minimal
$27$-dimensional $E_6$-module, Proposition~1 of \cite{Pevzner} states
that $(x_\alpha(t)-I)W$ is six-dimensional for every nontrivial element
$x_\alpha(t)$ of a root subgroup. Since this representation is
minuscule, $\rho(e_\alpha)^2=0$ and
$x_\alpha(t)-I=t\rho(e_\alpha)$. Hence $\rho(e_\alpha)$ has rank six.

All the displayed ranks are at least two. The omitted cases
$\Gr(1,N)\simeq\Gr(N-1,N)$ and $\mathbb S_3\simeq\PP^3$ are precisely the
projective-space cases.
\end{proof}

We now apply the graph equations of Proposition~\ref{prop:graph}
to the entries of Table~\ref{tab:ihss-list}.

\subsection{Grassmannians}

\begin{proposition}\label{prop:grassmannians}
A Grassmannian is a Heisenberg equivariant compactification if and only if it is an odd-dimensional projective space.
\end{proposition}

\begin{proof}
Write $X=\Gr(k,N)$ and put $A=(\C^{k})^*$,
$B=\C^{N-k}$, so $W=A\otimes B$. The cases $k=1$ and
$N-k=1$ are projective spaces. Assume $2\le k\le N-2$.
As $\dim X=k(N-k)$ is odd,  $\dim A,\dim B$ are odd and at least $3$.

Suppose the graph equations of Proposition~\ref{prop:graph} hold. For
$U\in\Gr(2,A)$ and $V\in\Gr(2,B)$, consider
\[
 \pi_{U,V}:A\otimes B\longrightarrow(A/U)\otimes(B/V),
 \qquad
 \ker\pi_{U,V}=A\otimes V+U\otimes B.
\]

Choose $\alpha\in A^*$, $\beta\in B^*$ satisfying
$(\alpha\otimes\beta)(z)\ne0$, and take any
$U\subset\ker\alpha$, $V\subset\ker\beta$. Then $\pi_{U,V}(z)\ne0$. Thus the pairs $(U, V)$ satisfying this condition form
a nonempty dense open subset of $\Gr(2, A) \times \Gr(2, B)$.

For $x=a\otimes b$ and $y=c\otimes d$, with
$a,c\in U$ and $b,d\in V$, the infinitesimal tensor action gives
\[
 \fl\cdot(a\otimes b)\subset A\otimes b+a\otimes B.
\]
Thus, both terms on the left of \eqref{eq:graph-torsion} lie in
$\ker\pi_{U,V}$. Applying $\pi_{U,V}$, and using
$\pi_{U,V}(z)\ne0$, gives
$\omega(a\otimes b,c\otimes d)=0$.
Decomposable tensors span $U\otimes V$, so
$\omega|_{U\otimes V}=0$. Since this is an algebraic condition in
$(U, V)$, it holds for every pair of two-planes. Any two
decomposable tensors lie in some $U\otimes V$, and decomposable
tensors span $A\otimes B$; hence $\omega=0$, contradicting
\eqref{eq:graph-radical}.
\end{proof}

\subsection{Quadrics}

\begin{proposition}\label{prop:quadrics}
A smooth  quadric $Q^n$ is a Heisenberg equivariant compactification  if and only if $n=1$ or $3$.
\end{proposition}

\begin{proof}
We may assume $n \geq 3$ is odd. Let $W$ be the $n$-dimensional cominuscule
 module, equipped with its nondegenerate symmetric form $q$.
Proposition~\ref{prop:graph} gives $f,z,\omega$, with
$f(W)\subset\mathfrak{so}(W,q)$, satisfying
\eqref{eq:graph-radical}--\eqref{eq:graph-closure}.

Define $A\in\mathfrak{so}(W,q)$ and $\zeta\in W^*$ by
\[
 q(Ax,y)=\omega(x,y),\qquad \zeta(x)=q(z,x).
\]
The nondegeneracy of $q$ and \eqref{eq:graph-radical} give
$\ker A=\C z$, in particular $Az=0$.

Define $T(x,y,w)=q(f(x)y,w)$.  As $f(x)\in\mathfrak{so}(W,q)$, one has
$T(x,y,w)=-T(x,w,y)$. Pairing
\eqref{eq:graph-torsion} with $w$, and then cyclically permuting
$(x,y,w)$, gives
\begin{equation}\label{eq:quadric-cyclic}
\begin{aligned}
 T(x,y,w)+T(y,w,x)&=\omega(x,y)\zeta(w),\\
 T(y,w,x)+T(w,x,y)&=\omega(y,w)\zeta(x),\\
 T(w,x,y)+T(x,y,w)&=\omega(w,x)\zeta(y).
\end{aligned}
\end{equation}
Adding the first and third identities in
\eqref{eq:quadric-cyclic} and subtracting the second yields
\[
 2q(f(x)y,w)
 =\omega(x,y)\zeta(w)-\omega(y,w)\zeta(x)
   +\omega(w,x)\zeta(y).
\]
Using $q(Au,v)=\omega(u,v)$, the right-hand side is the
$q$-pairing with $w$ of
$\omega(x,y)z-\zeta(x)Ay-\zeta(y)Ax$.
Since $q$ is nondegenerate, we obtain the following identity
\begin{equation}\label{eq:quadric-koszul}
 2f(x)y=\omega(x,y)z-\zeta(x)Ay-\zeta(y)Ax.
\end{equation}
Putting $x=z$ in \eqref{eq:quadric-koszul}, and using
$\omega(z,-)=0$, $Az=0$, and $\zeta(z)=q(z,z)$, we obtain the
following identity
\begin{equation}\label{eq:quadric-center}
 f(z)=-\frac{q(z,z)}2A.
\end{equation}

If $q(z,z)\ne0$, we have the $q$-orthogonal decomposition $W=\C z\oplus z^\perp$. As $\ker A=\C z$ and $A$ preserves $z^\perp$, the restricted operator $A|_{z^\perp}$ is invertible.
Thus $f(z)$ is not nilpotent, contradicting $f(z)\in\fn$.
Therefore
\begin{equation}\label{eq:quadric-isotropic-center}
 q(z,z)=0,\qquad f(z)=0.
\end{equation}
Condition \eqref{eq:graph-closure} now implies that all $f(x)$ commute.

Put $E=z^\perp$, and choose $w\in W$ with $q(z,w)=1$.
For $u,v\in E$, equation \eqref{eq:quadric-koszul} gives
\[
 2f(u)v=\omega(u,v)z,\qquad
 2f(w)v=\omega(w,v)z-Av,\qquad f(x)z=0.
\]
It then follows
\[
 0=[f(w),f(u)]v
   =\frac12f(u)Av
   =\frac14\omega(u,Av)z
   =\frac14q(Au,Av)z.
\]
Thus $A(E)$ is totally $q$-isotropic. On the other hand,
\[
 \dim A(E)=\dim E-\dim(\ker A\cap E)=n-2.
\]
But the isotropic subspace has maximal possible dimension $\left\lfloor\frac n2\right\rfloor$, which gives $ n-2\le\left\lfloor\frac n2\right\rfloor=\frac{n-1}{2}$, so $n\le3$, which gives $n=3$.

For existence, write $Q^3=B_2/P_1$, with $\alpha_1$ long and
$\alpha_2$ short. Let us consider the following elements: 
\[
 x=e_{\alpha_1}+e_{\alpha_2},\qquad
 y=e_{\alpha_1+\alpha_2},\qquad
 z=e_{\alpha_1+2\alpha_2}.
\]
Then we have some constant $c$ such that
\[
 [x,y]=cz\ne0,\qquad [x,z]=[y,z]=0,
\]
which implies that $x, y, z$ generate a three-dimensional Heisenberg algebra.
Their projections to the three-dimensional cominuscule module $W$ are the
three root vectors indexed by
$\alpha_1,\alpha_1+\alpha_2,\alpha_1+2\alpha_2$, hence form a basis.
The corresponding connected Heisenberg subgroup therefore has an open
orbit on $Q^3$.
\end{proof}


\subsection{Lagrangian Grassmannians}

\begin{proposition}\label{prop:lagrangians}
Among the Lagrangian Grassmannians, only $\LG(1,2) \simeq \mathbb{P}^1$ and
$\LG(2,4)\simeq Q^3$ are Heisenberg equivariant compactifications.
\end{proposition}

\begin{proof}
For $X=\LG(r,2r)$, put $A=(\C^{r})^*$, so
$W=\Sym^2A$. We may assume $r\ge3$ and, arguing by contradiction, that the graph equations hold.

For $U\in\Gr(2,A)$, let
\[
 \pi_U:\Sym^2A\longrightarrow\Sym^2(A/U),
 \qquad \ker\pi_U=U\odot A.
\]
Choose $\alpha\in A^*$ with $\alpha^2(z)\ne0$, and then choose
$U\subset\ker\alpha$ a two-plane. Then $\pi_U(z)\ne0$. For $a,b\in U$,
the infinitesimal action satisfies
\[
 \fl\cdot a^2\subset a\odot A.
\]
Applying $\pi_U$ to \eqref{eq:graph-torsion} with
$(x,y)=(a^2,b^2)$ gives $\omega(a^2,b^2)=0$.
Note that squares span $\Sym^2U$. The condition
$\omega|_{\Sym^2U}=0$ is closed on $\Gr(2,A)$. As it holds on a
dense open subset, it holds for every two-plane $U$. Any two squares lie in the symmetric square of
a two-plane, and squares span $\Sym^2A$. Therefore
$\omega=0$, a contradiction.

\end{proof}

\subsection{Spinor varieties}

\begin{proposition}\label{prop:spinors}
A spinor variety is a Heisenberg equivariant compactification if and 
only if it is one of the low-rank coincidences $\mathbb S_2\simeq\PP^1$ or $\mathbb S_3\simeq\PP^3$.
\end{proposition}

\begin{proof}
For $X=\mathbb S_r$, put $A=(\C^{r})^*$, so
$W=\Lambda^2A$. The low-rank coincidences $\mathbb S_2\simeq\PP^1, \mathbb S_3\simeq\PP^3, \mathbb S_4\simeq Q^6$
and the equality $\dim\mathbb S_5=10$ settle $r\le 5$.  We can now assume $r\ge6$ and the graph equations
hold in an odd-dimensional case.

For $K\in\Gr(3,A)$, let
\[
 \pi_K:\Lambda^2A\longrightarrow\Lambda^2(A/K),
 \qquad \ker\pi_K=K\wedge A.
\]
Choose $\alpha,\beta\in A^*$ with
$(\alpha\wedge\beta)(z)\ne0$. Since $r\ge6$, one can choose
$K\subset\ker\alpha\cap\ker\beta$, and then $\pi_K(z)\ne0$.
Every bivector in $\Lambda^2K$ is decomposable, and for a
decomposable $a\wedge b\in\Lambda^2K$ one has
\[
 \fl\cdot(a\wedge b)\subset\operatorname{span}(a,b)\wedge A
 \subset K\wedge A.
\]
Applying $\pi_K$ to \eqref{eq:graph-torsion} gives
$ \omega|_{\Lambda^2K}=0, $ which is a closed condition on ${\rm Gr}(3, A)$. Since it holds on a dense open subset, it holds for every three-plane $K$. 

For arbitrary $a,b,c,d\in A$, set
\[
\begin{aligned}
 x&=\omega(a\wedge b,c\wedge d),\\
 y&=\omega(a\wedge c,d\wedge b),\\
 z&=\omega(a\wedge d,b\wedge c).
\end{aligned}
\]
Apply $ \omega|_{\Lambda^2K}=0$ to the following three pairs of bivectors, each supported on a three-plane:
\[
\begin{aligned}
 &(a+c)\wedge b,\ (a+c)\wedge d,\\
 &(a+d)\wedge b,\ (a+d)\wedge c,\\
 &(a+b)\wedge c,\ (a+b)\wedge d.
\end{aligned}
\]
Bilinear expansion gives
\[
 x+z=0,\qquad x+y=0,\qquad y+z=0.
\]
Hence $x=0$. Since the decomposable bivectors span
$\Lambda^2A$, we obtain $\omega=0$, a contradiction.
\end{proof}

\subsection{Exceptional cases}

The two Cayley planes $E_6/P_1$ and $E_6/P_6$ have dimension
$16$, so parity excludes them. It remains to treat the
$27$-dimensional Freudenthal variety $E_7/P_7$. We use the following standard facts (see for example \cite[Part II, Section~4]{McCrimmon}) about cubic Jordan algebras, with the normalizations fixed below.

Let $\mathbb O_{\C}$ be the complexified octonion algebra, with its
standard conjugation, and let $J=H_3(\mathbb O_{\C}) $ be the Jordan algebra of $3\times3$ Hermitian matrices over
$\mathbb O_{\C}$, endowed with the Jordan product
$x\circ y:=(xy+yx)/2$. This is the $27$-dimensional simple
cubic Jordan algebra called the \emph{complex Albert algebra}. Its
unit is denoted by $e$, its Jordan trace by $\operatorname{Tr}$,
and  $ \langle x,y\rangle:=\operatorname{Tr}(x\circ y)$
is a nondegenerate symmetric bilinear form. The  norm
$N: J\to\C$ is the cubic determinant on Hermitian matrices. We denote
by $N(x,y,z)$ its symmetric trilinear polarization, normalized by
$N(x,x,x)=N(x)$.

For $x \in J$, its \emph{quadratic adjoint} $x^\#\in J$ is uniquely characterized by
\[
  dN_x(y)=\langle x^\#,y\rangle
  \qquad \forall y\in J,
\]
which satisfies
\[
  x\circ x^\#=N(x)e,
  \qquad (x^\#)^\#=N(x)x.
\]
Its polarization is the symmetric bilinear product
\[
  x\times y=(x+y)^\#-x^\#-y^\#
            =d(x^\#)_x(y).
\]
With the preceding normalization, one has
\begin{equation}\label{eq:albert-basic-identities}
  \langle x^\#,y\rangle=3N(x,x,y),\qquad
  \langle x\times v,w\rangle=6N(x,v,w).
\end{equation}
A nonzero element $x$ has Jordan rank one precisely when
$x^\#=0$. So the projectivized rank-one locus is the
Cayley plane $E_6/P_1\subset\PP J$. The structure group of $J$
preserves $N$ up to a scalar, and its Lie algebra is
$\mathfrak{s}=\mathfrak e_6\oplus\C$. In particular, it
preserves the affine rank-one cone. 

\begin{proposition}\label{lem:albert}
Let $J$ be the $27$-dimensional complex Albert algebra and let
$\mathfrak s=\mathfrak e_6\oplus\C$ be its structure algebra in the standard cominuscule representation. 
There does not exist a triple $(f, z, \omega)$ satisfying the following conditions,  where $f$ is a linear map
$f:J\to\mathfrak s$, $z\in J$ a nonzero element and
$\omega\in\Lambda^2J^*$:
\begin{equation}\label{eq:albert-torsion}
 \ker\omega=\C z,\qquad
 f(x)y-f(y)x=\omega(x,y)z\quad \forall x,y\in J.
\end{equation}
\end{proposition}

\begin{proof}
The affine cone $\widehat{\cC}\subset J$ of rank-one elements is
defined by $x^\#=0$. It is the cone over the Cayley plane (see for example \cite{LM}), hence, it is
irreducible; its nonzero locus is smooth and its span is a
nonzero $E_6$-submodule of the irreducible module $J$, hence it spans $J$.

The polarized cubic defines an $E_6$-equivariant map
\begin{equation}\label{eq:albert-cubic-map}
 J\longrightarrow\Sym^2J^*,\qquad
 \zeta\longmapsto N(\zeta,-,-).
\end{equation}
This map is nonzero, so the irreducibility of $J$ makes it injective.

Suppose that $f,z,\omega$ satisfy \eqref{eq:albert-torsion}. For
nonzero rank-one elements $a,b$, the structure algebra
$\mathfrak s$ preserves $\widehat{\cC}$, and therefore
\[
 f(a)b\in T_b\widehat{\cC},\qquad
 f(b)a\in T_a\widehat{\cC}.
\]
Differentiating $x^\#=0$ in a tangent direction
$v\in T_x\widehat{\cC}$ gives $x\times v=0$. By
\eqref{eq:albert-basic-identities},
\[
 N(b,f(a)b,a)=0,\qquad N(a,f(b)a,b)=0.
\]
The symmetry of $N$, followed by \eqref{eq:albert-torsion}, yields
\begin{equation}\label{eq:albert-vanishing-product}
 0=N(f(a)b-f(b)a,a,b)
   =\omega(a,b)N(z,a,b).
\end{equation}

By the injectivity of \eqref{eq:albert-cubic-map}, together with the fact
that rank-one elements span $J$,  the bilinear form
$(a,b)\mapsto N(z, a, b)$ is not identically zero for the fixed nonzero vector $z$. Its nonvanishing locus on
$\widehat{\cC}_{\mathrm{sm}}\times \widehat{\cC}_{\mathrm{sm}}$
is therefore a nonempty dense open subset.
Then equation \eqref{eq:albert-vanishing-product} forces
$\omega$ to vanish on that dense open subset, hence on the product of the two rank-one cones. 
Since the cone spans $J$, we have $\omega=0$, contradicting $\ker\omega=\C z$.
\end{proof}

\subsection{Completion of the classification for $G$ simple}

\begin{proof}[Proof of Theorem~\ref{thm:main} for $G$ simple]
Assume that $X=G/P$ is a Heisenberg equivariant compactification with $G$ simple.
By Proposition~\ref{prop:reduction}, $X$ is either an adjoint variety
or a cominuscule variety. In the latter
case, we  apply Proposition~\ref{prop:graph} to the cases in
Table~\ref{tab:ihss-list}. Proposition~\ref{prop:grassmannians} leaves only
odd-dimensional projective spaces among Grassmannians.
Propositions~\ref{prop:lagrangians}, \ref{prop:spinors}, and
\ref{prop:quadrics} leave respectively $Q^3$, the projective-space
coincidence $\mathbb S_2\simeq\PP^1, \mathbb S_3\simeq\PP^3$, and $Q^3$. Parity excludes
the two $E_6$-varieties, while Proposition~\ref{lem:albert} excludes
$E_7/P_7$. This exhausts the table, which gives
\[
 X\simeq\PP^{2m+1}\quad(m\ge1),\qquad\text{or}\qquad X\simeq Q^3.
\]
The projective space $\PP^{2m+1}$ is the adjoint variety of the simple
group $\operatorname{Sp}_{2m+2}$. Therefore, the only non-adjoint
possibility is $Q^3$.
\end{proof}

\section{The case where $G$ is semisimple} \label{sec:nonsimple}

We start with a construction which provides many product examples, including those in Theorem \ref{thm:main}(ii)), of 
Heisenberg equivariant compactifications. This gives a generalization of the construction in \cite[Proposition 5.5]{FuHwangContact}.

\begin{proposition}\label{prop:nonsimple-family}
Let $Y$ be an equivariant compactification of the vector group
$\mathbb G_a^d$. For every integer $m\geq d$, the product
$ \PP^{2m+1-d} \times Y$
is an equivariant compactification of the $(2m+1)$-dimensional
Heisenberg group $H_{2m+1}$.
\end{proposition}

\begin{proof}
Put $k=m-d\geq0$. Choose vector spaces $E$ and $F$ with
$\dim E=m$ and $\dim F=k$, together with a linear surjection
$\varphi:E^*\to F$. If $k=0$, take $F=0$ and $\varphi=0$. Write the
Heisenberg algebra as $\fh=E^*\oplus E\oplus\C z$ with Lie brackets given by:
\[
  [\alpha,b]=\alpha(b)z, [\fh, z]=0, [\alpha, \alpha'] = [b, b']=0, \forall \alpha, \alpha' \in E^*, b, b'\in E.
\]

Let $V=\C e_0\oplus E\oplus\C e_\infty\oplus F$. We now define a map $\rho:\fh\to\mathfrak{gl}(V)$ and then we will show this gives a representation of $\fh$ on $V$. For any $\alpha \in E^*, b \in E$, define endomorphisms of $V$ by
\[
\begin{aligned}
  \rho(\alpha)e_\infty&=\varphi(\alpha),&
  \rho(\alpha)v&=\alpha(v)e_0 && \forall v\in E,\\
  \rho(b)e_\infty&=b,&
  \rho(z)e_\infty&=e_0,
\end{aligned}
\]
and let them vanish on every summand on which their values have not
been specified. Direct calculation gives
\[
  [\rho(\alpha),\rho(b)]=\alpha(b)\rho(z),
\]
while all remaining defining brackets vanish. Hence
$\rho:\fh\to\mathfrak{gl}(V)$ is a representation. It is faithful:
if
$\rho(\alpha)+\rho(b)+c\rho(z)=0$, evaluation on arbitrary
$v\in E$ gives $\alpha=0$, and evaluation on $e_\infty$ then
gives $b=0=c$.

One checks that  $\rho(x)^3=0$ for all $x \in \fh$. The polynomial
exponential therefore integrates $\rho$ to a faithful algebraic
representation
\[
  R:H_{2m+1}\longrightarrow\operatorname{GL}(V).
\]
The induced projective representation is faithful as well, since a
scalar unipotent matrix is the identity. Moreover,  $\dim\PP(V)=m+k+1=2m+1-d.$

Let $p_\infty=[e_\infty] \in \mathbb{P}V$. We now compute its stabilizer in $H_{2m+1}$. For
$x=\alpha+b+cz\in\fh$, one has
\[
\begin{aligned}
  \rho(x)e_\infty&=\varphi(\alpha)+b+ce_0,\\
  \rho(x)^2e_\infty&=\alpha(b)e_0,\\
  \rho(x)^3e_\infty&=0.
\end{aligned}
\]
From which we deduce
\[
  \exp(\rho(x))e_\infty
  =e_\infty+\varphi(\alpha)+b+
    \left(c+\frac12\alpha(b)\right)e_0.
\]
It follows that $\exp(\rho(x)) \in \operatorname{Stab}_{H_{2m+1}}(p_\infty)$ if and only if  $\varphi(\alpha)=0$, $b=0$, and $c=0$. 
Thus $\operatorname{Stab}_{H_{2m+1}}(p_\infty)=\exp(S)$ with $ S:=\ker \varphi$, and $\dim S=m-k=d$.

As $Y$ is an equivariant compactification of $\mathbb G_a^d$, there exists an open subset $U_Y$ of $Y$ equivariantly isomorphic to $\mathbb{G}_a^d$. For  any point $y_0 \in U_Y$, its stabilizer is trivial.

Choose an isomorphism $\tau:S\to\Lie(U_Y)$ and a linear projection $\operatorname{pr}_S:E^*\to S$. Define
\[
  q:H_{2m+1}\longrightarrow U_Y,
  \qquad
  q\bigl(\exp(\alpha+b+cz)\bigr)
  =\exp_{U_Y}\!\bigl(\tau(\operatorname{pr}_S\alpha)\bigr).
\]
This is an algebraic group homomorphism. Indeed, in exponential coordinates, the Baker--Campbell--Hausdorff correction belongs to the
central line $\C z$, which $q$ kills. Its restriction
\[
  q|_{\exp(S)}:\exp(S)\xrightarrow{\sim}U_Y
\]
is an isomorphism.

Let $H_{2m+1}$ act diagonally on
$Y\times\PP(V)$, through $q$ on $Y$ and through $R$ on
$\PP(V)$. Since $U_Y$ acts simply transitively at $y_0$,
\[
  \operatorname{Stab}_{H_{2m+1}}(y_0)=\ker q.
\]
It follows that
\[
  \operatorname{Stab}_{H_{2m+1}}(y_0,p_\infty)
  =\ker q\cap\exp(S)=\{e\}.
\]
Therefore, the orbit through $(y_0,p_\infty)$ has dimension
\[
  2m+1=d+(2m+1-d)=\dim\bigl(Y\times\PP(V)\bigr),
\]
and hence is open. Its stabilizer is trivial, so the orbit map
identifies this open orbit equivariantly with $H_{2m+1}$, which concludes the proof.
\end{proof}

We next prove that if a product of rational homogeneous spaces is a Heisenberg equivariant compactification, then it must be of the form in the previous Proposition.


Recall that a contact structure on a complex manifold $M$ is a corank-one subbundle $F\subset TM$ whose Levi bracket
\[
\Lambda^2F\longrightarrow TM/F,\qquad
(v,v') \mapsto [v,v']\bmod F,
\]
is everywhere nondegenerate. Adjoint varieties are standard examples of projective contact manifolds.

\begin{lemma}\label{lem:contact}
Let $H$ be a  Heisenberg group, with Lie algebra $\mathfrak h$ satisfying $ [\mathfrak h,\mathfrak h]=Z(\mathfrak h)=\C z.$
Suppose that $H$ acts on a complex contact manifold $M$ with an open orbit, preserving its
contact distribution $F\subset TM$. Then the
infinitesimal stabilizer at every point of the open orbit is zero. In
particular, $\dim H=\dim M$.
\end{lemma}

\begin{proof}
Fix a point $x$ of the open orbit. Let $\mathfrak s:=\ker (\mathfrak h\longrightarrow T_xM )$ be its infinitesimal stabilizer, and put
\[
 \mathfrak d:=\{\xi\in\mathfrak h:\xi_M(x)\in F_x\}.
\]
Since the orbit is open, the infinitesimal orbit map is surjective. Thus $\mathfrak s \subset \mathfrak d$, $ \mathfrak d/\mathfrak s\simeq F_x$, and $\mathfrak d$ is a hyperplane in $\mathfrak h$.

The invariance of $F$ under the isotropy at $x$ gives $[\mathfrak s,\mathfrak d]\subset\mathfrak d.$
Moreover, under the identification
$\mathfrak d/\mathfrak s\simeq F_x$, the Levi form of $F$ is induced, up to sign, by
\[
 \Lambda^2(\mathfrak d/\mathfrak s) \longrightarrow \mathfrak h/\mathfrak d,
 \qquad
 (\bar\xi,\bar\eta)\longmapsto [\xi,\eta]\bmod\mathfrak d.
\]

If $z\in\mathfrak d$, then
$ [\mathfrak h,\mathfrak h]=\C z\subset\mathfrak d$, so the above Levi form vanishes identically, contradicting the contact condition. 
Therefore $z\notin\mathfrak d$, and hence
$\mathfrak h=\mathfrak d\oplus\C z.$
So projection onto $\mathfrak h/\C z$ restricts to an
isomorphism $\mathfrak d\xrightarrow{\ \sim\ }\mathfrak h/\C z.$
The Heisenberg bracket induces a nondegenerate alternating form on
$\mathfrak h/\C z$; therefore its restriction $\Lambda^2\mathfrak d\longrightarrow\C z $ is nondegenerate.

On the other hand, $[\mathfrak s,\mathfrak d]\subset\mathfrak d$ and the Heisenberg
identity give
\[
 [\mathfrak s,\mathfrak d]
 \subset
 \mathfrak d\cap[\mathfrak h,\mathfrak h]
 =\mathfrak d\cap\C z
 =0.
\]
Thus $\mathfrak s\subset\mathfrak d$ lies in the radical of the
nondegenerate alternating form on $\mathfrak d$. Hence
$\mathfrak s=0.$
The dimension assertion follows from the openness of the orbit.
\end{proof}

\begin{proposition}\label{prop:semisimple-reduction}
Let $X=G/P$ be a projective rational homogeneous variety of dimension $2m+1$ and write $X$ as a product $X=\prod_{i=1}^sX_i$ with $s\geq2$,
where $G_i:=\Aut^0(X_i)$ is simple and $X_i\simeq G_i/P_i$. Suppose
that $X$ is an equivariant compactification of a Heisenberg group $H$. Let
$\rho_i:H\to G_i$ be the homomorphism induced by the action on the
$i$-th factor. Then:
\begin{enumerate}
\item $H$ has an open orbit on every $X_i$;
\item there is a unique index $i_0$ for which $\rho_{i_0}$ is
injective, and $2\dim X_{i_0}>\dim X$;
\item every $X_i$ is cominuscule;
\item $X_{i_0}$ is a projective space.
\end{enumerate}
\end{proposition}

\begin{proof}
Write
$\fh=V\oplus\C z$, where $\dim V=2m$ and
$[v,w]=\omega(v,w)z$ for a nondegenerate alternating form $\omega$.
Choose $x=(x_1,\ldots,x_s)$ in the free open $H$-orbit. For each $i$,
let $d_i:\fh\to T_{x_i}X_i$ be the infinitesimal orbit map. Since
$d=\bigoplus_i d_i:\fh\to T_xX$ is an isomorphism,
\[
 2m+1=\rank d\leq\sum_i\rank d_i
   \leq\sum_i\dim X_i=2m+1.
\]
Thus $\rank d_i=\dim X_i$ for every $i$, which proves (1).

We claim that every nonzero ideal $I\subset\fh$ contains $\C z$. This is trivial if $I \subset \C z$.  Now assume 
$I\not\subset\C z$, choose $v+az\in I$ with $v\ne0$ and then choose
$w\in V$ with $\omega(v,w)\ne0$; thus
$0\ne[v+az,w]\in I\cap\C z$, proving the claim.
Suppose that $\rho_i$ is not injective.
Then $\ker(d\rho_i)\ne0$: otherwise $\ker\rho_i$ would be finite, while
a connected unipotent group over $\C$ has no nontrivial finite subgroup.
It follows that $\C z\subset\ker(d\rho_i)$, so
$A_i:=\rho_i(H)$ is a connected abelian unipotent group, hence a vector
group.

The open $A_i$-orbit through $x_i$ is free. In fact, if
$a\in A_i$ fixes $x_i$, then commutativity gives
$a(bx_i)=b(ax_i)=bx_i$ for all $b\in A_i$. Thus $a$ fixes the open
orbit pointwise and therefore acts trivially on $X_i$. Since $A_i$ is
the effective image of $H$ in $G_i$, we have $a=e$. It follows that
$X_i$ is an equivariant compactification of
$\mathbb G_a^{\dim X_i}$ and is cominuscule by
\cite[Theorem~1]{Arzhantsev}.

There is at least one injective $\rho_i$; otherwise $d(z)=0$. Let
$\rho_j$ be injective and put $\mathfrak s_j:=\ker d_j$. We have
$\mathfrak s_j\cap\C z=0$. In fact, if $z\in\mathfrak s_j$, then the
central subgroup $\exp(\C z)$ fixes the open orbit $H\cdot x_j$
pointwise and hence lies in $\ker\rho_j$. It follows that
$[\mathfrak s_j,\mathfrak s_j]=0$. The quotient map
$\fh\to\fh/\C z\simeq V$ is injective on $\mathfrak s_j$, and its
image is $\omega$-isotropic. Hence $\dim\mathfrak s_j\leq m$ and
\[
 \dim X_j=2m+1-\dim\mathfrak s_j\geq m+1>\frac12\dim X.
\]
There can therefore be only one injective factor; denote it by $i_0$.
This proves (2), and the preceding paragraph already shows that
$X_i$ is cominuscule for $i\ne i_0$.

Let $\fg=\bigoplus_i\fg_i$ and $\fp=\bigoplus_i\fp_i$, where
$\fg_i=\Lie(G_i)$ and $\fp_i=\Lie((G_i)_{x_i})$. The diagonal action
gives $\fg=\fh\oplus\fp$ as vector spaces. Applying the proof of
Lemma~\ref{lem:filtered-complement}, which applies verbatim  to the direct sum of the fundamental
gradings gives
\[
 \sum_i\dim[(\fg_i)_-,(\fg_i)_-]
 =\dim[\fg_-,\fg_-]\leq1.
\]
By Lemma~\ref{lem:parabolic-dichotomy}, every $X_i$ is cominuscule or
adjoint, with at most one adjoint factor. The only possible adjoint
factor is $X_{i_0}$. If $X_{i_0}$ were an adjoint variety different from projective spaces,
$G_{i_0}=\Aut^0(X_{i_0})$ would preserve its homogeneous contact
distribution. Lemma~\ref{lem:contact} would give
$\dim H=\dim X_{i_0}$, contradicting $s\geq2$. With our convention
$G_i=\Aut^0(X_i)$, a projective space is in the cominuscule
presentation. Thus every $X_i$ is cominuscule, proving (3).

It remains to prove (4). Put $\mathfrak s:=\mathfrak s_{i_0}$. The
infinitesimal orbit map identifies
$\fh/\mathfrak s$ with $T_{x_{i_0}}X_{i_0}$. For
$u\in\mathfrak s$, the infinitesimal isotropy operator is
$\bar v\longmapsto\overline{[u,v]}.$
Its image is contained in $\C\bar z$, and it kills $\bar z$ because
$z$ is central. Hence its square is zero and its rank is at most one.


Apply Lemma~\ref{lem:standard-position} to the open action of
$\rho_{i_0}(H)$ on $X_{i_0}$. After conjugation, we may choose a
Borel subgroup $B=TU\subset G_{i_0}$ and a $T$-fixed point
$o\in X_{i_0}$ such that $\rho_{i_0}(H)\subset U$ and $o$ belongs
to the open $\rho_{i_0}(H)$-orbit.

Let
\[
 \fg_{i_0}=\fg_{-1}\oplus\fg_0\oplus\fg_1,
 \qquad \fp_{i_0}=\fg_0\oplus\fg_1,
\]
be the grading defined by the stabilizer of $o$, using the Borel
opposite to $B$. Then
$\Lie(U)=\fn\ltimes\fg_{-1}$, where $\fn=\Lie(U)\cap\fg_0$. Therefore
\[
 d\rho_{i_0}(\fh)\subset\fn\oplus\fg_{-1},
 \qquad d\rho_{i_0}(\mathfrak s)\subset\fn\subset\fg_0.
\]
Under the identifications
$\fh/\mathfrak s\simeq T_oX_{i_0}\simeq\fg_{-1}$, the action of
$d\rho_{i_0}(\mathfrak s)$ on $\fg_{-1}$ is precisely the
infinitesimal isotropy action above. Every element of $\fn$ acts
nilpotently. If $X_{i_0}$ were not a projective space,
Lemma~\ref{lem:rank-one} would force every such rank-at-most-one operator
to vanish. The $\fg_0$-action on $\fg_{-1}$ is faithful, so
$d\rho_{i_0}(\mathfrak s)=0$. Since $d\rho_{i_0}$ is injective,
$\mathfrak s=0$. But
\[
 \dim\mathfrak s=\dim H-\dim X_{i_0}
 =\sum_{i\ne i_0}\dim X_i>0,
\]
a contradiction. Hence $X_{i_0}$ is a projective space.
\end{proof}

\begin{proof}[Proof of Theorem~\ref{thm:main}]
The necessity for a simple connected automorphism group was proved in
Section~3, and the remaining necessity follows from
Proposition~\ref{prop:semisimple-reduction}. Conversely, adjoint
varieties are compactifications of the Heisenberg group defined by their
contact gradings, and $Q^3$ is covered by
Proposition~\ref{prop:quadrics}. Let
$X\simeq\PP^{2m+1-d}\times Y$ be as in Theorem~\ref{thm:main}. By
\cite[Theorem~1]{Arzhantsev}, the product $Y$ is an equivariant
compactification of $\mathbb G_a^d$. Since $d\leq m$,
Proposition~\ref{prop:nonsimple-family} supplies the required
$H_{2m+1}$-action on $X$.
\end{proof}

\end{document}